\documentclass[reqno]{amsart}
\usepackage{amsmath,verbatim,amsthm,amssymb,accents,tensor}
\usepackage[all]{xy}
\usepackage{tikz-cd}
\allowdisplaybreaks[4]
\usepackage[latin1]{inputenc}
\allowdisplaybreaks[4]
\newtheorem{theorem}{Theorem}[section]
\newtheorem{proposition}[theorem]{Proposition}

\newtheorem{definition}[theorem]{Definition} 
\newtheorem{lemma}[theorem]{Lemma} 
\theoremstyle{definition} 
\newtheorem{remark}[theorem]{Remark}

\numberwithin{equation}{section} 
\newcommand\x[1]{\mathrm{#1}}
\renewcommand\o[1]{\overline{#1}}

\newcommand\h[1]{\hat{#1}}

\renewcommand\r[1]{\mathring{#1}}

\renewcommand\t[1]{\tilde{#1}}

\newcommand\f[2]{\frac{#1}{#2}}
\newcommand\I{\int\limits}
\renewcommand\({\Big(}
\renewcommand\){\Big)}

\renewcommand\]{\Big]}

\renewcommand\>{\rangle}
\renewcommand\S{\sum\limits}
\renewcommand\P{\prod\limits}
\newcommand\ba[2]{\begin{pmatrix}{#1}\\{#2}\end{pmatrix}}
\newcommand\ab[2]{\begin{pmatrix}{#1}&{#2}\end{pmatrix}}
\newcommand\bb[4]{\begin{pmatrix}{#1}&{#2}\\{#3}&{#4}\end{pmatrix}}

\newcommand\F[1]{\sqrt{#1}}

\newcommand\mL{\mathfrak m}

\newcommand\AL{\mathcal A}
\newcommand\BL{\mathcal B}

\newcommand\EL{\mathcal E}

\newcommand\IL{\mathcal I}
\newcommand\KL{\mathcal K}
\newcommand\LL{\mathcal L}
\newcommand\ML{\mathcal M}

\newcommand\PL{\mathcal P}

\newcommand\Bl{\mathbf B}
\newcommand\Cl{\mathbf C}
\newcommand\Dl{\mathbf D}

\newcommand\Nl{\mathbf N}
\newcommand\Ol{\mathbf O}
\newcommand\Pl{\mathbf P}

\newcommand\la{\alpha}
\newcommand\lb{\beta}

\renewcommand\lg{\gamma}
\newcommand\lG{\Gamma}

\newcommand\lD{\Delta}

\newcommand\li{\iota}
\newcommand\lk{\kappa}
\newcommand\lf{\phi}

\newcommand\lF{\Phi}
\renewcommand\ll{\lambda}
\newcommand\lL{\Lambda}
\newcommand\lm{\mu}
\renewcommand\ln{\nu}

\newcommand\lO{\Omega}
\newcommand\lp{\pi}
\newcommand\lr{\rho}
\renewcommand\lq{\psi}
\newcommand\lQ{\Psi}
\newcommand\ls{\sigma}

\newcommand\Lt{\tau}
\newcommand\lt{\vartheta}

\newcommand\lx{\xi}
\newcommand\lz{\zeta}
\newcommand\m{{\boldsymbol m}}
\renewcommand\k{{\boldsymbol k}}
\newcommand\n{{\boldsymbol n}}

\newcommand\je{{\boldsymbol 1}}
\renewcommand\.{\cdot}
\newcommand\dl{\partial}

\newcommand\oo{\infty}
\newcommand\oc{\circ}

\newcommand\op{\oplus}

\newcommand\xx{\times}

\newcommand\xt{\otimes}

\newcommand\ic{\subset}

\newcommand\ui{\cap}

\newcommand\yi{\wedge}

\renewcommand\l{\ell}
\renewcommand\:{\cdots}
\renewcommand\;{\ldots}
\newcommand\wt{\widetilde}
\newcommand\er{\eqref} 

\newcommand\zl{\boldsymbol z}
\newcommand\be[1]{\begin{equation}\label{#1}}
\newcommand\ee{\end{equation}}
\newcommand\bi{\begin{itemize}}
\newcommand\ei{\end{itemize}}
\renewcommand\i[1]{\item{#1}}

\begin{document}


\title[Singular Hilbert modules]{Singular Hilbert modules on Jordan-Kepler varieties}
\author[G. Misra, H. Upmeier]{Gadadhar Misra {\rm and} Harald Upmeier} 
\medskip

\address{Department of Mathematics, Indian Institute of Science, Bangalore 560012, India}
\email{gm@iisc.ac.in}

\address{Fachbereich Mathematik, Universit\"at Marburg, D-35032 Marburg, Germany}
\email{upmeier{@}mathematik.uni-marburg.de}

\thanks{The first-named author was supported by the J C Bose National Fellowship of the DST and CAS II of the UGC and the second-named author was supported by an Infosys Visiting Chair Professorship at the Indian Institute of Science}
\medskip



\begin{abstract} We study submodules of analytic Hilbert modules defined over certain algebraic varieties in bounded symmetric domains, the so-called Jordan-Kepler varieties $V_\l$ of arbitrary rank $\l.$ For $\l>1$ the singular set of $V_\l$ is not a complete intersection. Hence the usual monoidal transformations do not suffice for the resolution of the singularities. Instead, we describe a new higher rank version of the blow-up process, defined in terms of Jordan algebraic determinants, and apply this resolution to obtain the rigidity of the submodules vanishing on the singular set.
\end{abstract}

\maketitle

\setcounter{section}{-1}

\section{Introduction}
R. G. Douglas introduced the notion of Hilbert module $\ML$ over a function algebra $\AL$ and reformulated several questions of multi-variable operator theory in the language of Hilbert modules. Having done this, it is possible to use techniques from commutative algebra and algebraic geometry to answer some of these questions. One of the very interesting examples is the proof of the Rigidity Theorem for Hilbert modules \cite[Section 3]{DPSY}, which we discuss below. 

A Hilbert module is a complex separable Hilbert space $\ML$ equipped with a multiplication
$$\mL:\AL\to\BL(\ML),\,\, \mL_p(f)=p\.f,\,\, f\in\ML,\, p\in\AL,$$ 
which is a continuous algebra homomorphism. Here $\BL(\ML)$ denotes the algebra of all bounded linear operators on $\ML$. The continuity of the module multiplication means
$$\|\mL_p f\| \le C\, \|f\|, f\in\ML,  p\in\AL$$
for some $C>0$. Familiar examples are the Hardy and Bergman spaces defined on bounded  domains in $\Cl^d$. Sometimes, it is convenient to consider the module multiplication over the polynomial ring $\Cl[\zl]$ in $d$ variables rather than a function algebra. In this case, we require that 
$$\|\mL_p f\| \le C_p\, \|f\|, f\in\ML,  p\in\AL$$
for some $C_p>0.$ We make this ``weak'' continuity assumption through out the paper.  

In  what follows, we will consider a natural class of Hilbert modules consisting of holomorphic functions, taking values in $\Cl^n$,  defined on a bounded domain $\lO\subseteq\Cl^d$. Thus \textsf{(i)} we assume $\ML\subseteq\x{Hol}(\lO,\Cl^n)$. A second assumption \textsf{(ii)} is to require that the evaluation functional 
$$\mbox{\rm ev}_z:\ML\to\Cl^n,\quad\mbox{ev}_z(f):=f(z),$$ 
is continuous and surjective, see \cite[Definition 2.5]{AS}. Set 
$$\KL(z,w):=\mbox{\rm ev}_z\mbox{\rm ev}^*_w:\lO\xx\lO\to\Cl^{n\xx n}.$$
The function $\KL,$ which is holomorphic in the first variable and anti-holomorphic in the second variable is called the 
{\bf reproducing kernel} of the Hilbert module $\ML.$ A further assumption \textsf{(iii)} is that $\Cl[\zl]\subseteq\ML$ is dense in 
$\ML$. A Hilbert module with these properties is said to be an {\bf analytic Hilbert module}. In this paper, we study a class of Hilbert modules which are submodules of analytic Hilbert modules.

From the closed graph theorem, it follows that $\mL_p f\in\ML$ for any $f\in\ML$ and $p\in\Cl[\zl]$. Also, the density of the polynomials implies that the eigenspace $\ker\,(\mL_p-p(w))^*$ is spanned by the vectors 
$$\KL_w(\.)\lz:=\KL(\.,w)\lz$$ 
for $\lz\in\Cl^n$, i.e.,
$$\ker\ (\mL_p-p(w))^*=\x{Ran}\ \KL_w,$$ 
see \cite[Remark, p. 285]{DM}. Since the metric $\KL(w,w)$ is invertible by our assumption, it follows that the dimension of the kernel $\{\KL_w(\.)\lz:\lz\in\Cl^n\}$ is exactly $n$ for all $w\in\lO$. Clearly, the map $w\mapsto\KL_w(\.)\lz$, $\lz\in\Cl^n$ is a holomorphic map on $\lO^*:=\{w\in\Cl^d:\o w\in\lO\}$. It serves as a holomorphic section of the trivial vector bundle 
$$\EL:=\{(w,v):\ w\in\lO^*,\, v\in\ker\ (\mL_p-p(w))^*\}\subseteq\lO^*\xx\ML$$
with fibre
$$\EL_w=\ker\ (\mL_p-p(w))^*=\x{Ran}\ \KL_w, \,w\in\lO^*.$$
A refinement of the argument given in \cite{AS} (which, in turn, is an adaptation of ideas from \cite{CD}), then shows that the isomorphism class of the module $\ML$ and the equivalence class of the holomorphic Hermitian bundle $\EL$ determine each other. The case $d=1,$ originally considered in \cite{CD}, corresponds to Hilbert modules over the polynomial ring in one variable. The proof in \cite{CD}, in this particular case, has a slightly different set of hypotheses. In the paper \cite{CD}, among other things, a complete set of invariants for the equivalence class of $\EL$ is given. If $n=1$, as is well known, this is just the curvature of the holomorphic line bundle $\EL$.

There is a natural notion of module isomorphism, namely, the existence of a unitary {\color{blue} linear} map $U:\ML\to\wt\ML$, which intertwines the module multiplications $\mL_p$ and $\t\mL_p$, that is, 
$$U\mL_p=\t\mL_p U.$$
Clearly, a Hilbert module $\ML$ over the polynomial ring $\Cl[\zl]$ is determined by the commuting tuple of multiplication by the coordinate functions on $\ML$ and vice-versa. Thus the notion of module isomorphism corresponds to the usual notion of unitary equivalence of two such $d$-tuples of multiplication operators by a fixed unitary. If $\lG:\ML_1\to\ML_2$ is a module map, then it maps the eigenspace of $\ML_1$ at $w$ into that of $\ML_2$ at $w$. Thus $\lG(\KL^1(\.,w)\lz)\subseteq\{\KL^2(z,w)\lx:\lx\in\Cl^n\}$, where 
$\KL^i$ are the reproducing kernels of the Hilbert modules $\ML_i$, i$=1,2,$ respectively. Hence we obtain a holomorphic map 
$\lF_\lG:\lO\to\Cl^{n\xx n}$ with the property 
$$\lG\KL^1(z,w)=\lF_\lG(w)^*\KL^2(z,w)$$
for any fixed but arbitrary $w.$ Thus any module map between two analytic Hilbert modules is induced by a holomorphic matrix-valued function $\lF:\lO\to\Cl^{n\xx n}$, see \cite[Theorem 3.7]{CS}. Moreover, if the module map is invertible, then $\lF_\lG(z)$ must be invertible. Finally, if the module map is assumed to be unitary, then 
$$\KL^1(z,w)=\lF_\lG(z)\ \KL^2(z,w)\ \lF_\lG^*(w)$$
for all $z,w\in\lO.$  

Let us describe an instance of the Sz.-Nagy -- Foias theory in the language of Hilbert modules following \cite{DP}. Let $T$ be a contraction on some Hilbert space $\ML$. The module multiplication determined by this operator is the map 
$\mL_p(f)=p(T)f$, $p\in\Cl[z]$, $f\in\ML$. From the contractivity of $T$, it follows that $\|\mL_p\|\le\|f\|$ and in this case,  the Hilbert module $\ML$ is said to be contractive. Now, assume that ${T^*}^n\to 0$ as $n\to\oo$. Then Sz.-Nagy -- Foias show that there exists an isometry $R$ and a co-isometry $R^\prime$ such that, for the unit disk $\Dl,$ the sequence
\begin{eqnarray*}\label{e1}\xymatrix{0\ar[r]& H_\EL^2(\Dl)\ar[r]^{R}&H_{\EL^\prime}^2(\Dl)\ar[r]^{R^\prime}&\ML\ar[r]&0},\end{eqnarray*}
where $\EL$ and $\EL^\prime$ are a pair of (not necessarily finite dimensional) Hilbert spaces, is exact. The map $R$ is essentially the characteristic function of the contraction $T$ and serves to identify the contractive module $\ML$ as a quotient module of 
$H_{\EL^{\prime}}^2(\Dl)$ by the image of $H_{\EL}^2(\Dl)$ under the isometric map $R$. 

For any planar domain $\lO,$ a model theory for completely contractive Hilbert modules over the function algebra $\x{Rat}(\lO),$ consisting of rational functions with poles off the closure $\o\lO$, has been developed by Abrahamse and Douglas in the paper \cite{AD}. However, the situation is much more complicated for Hilbert modules over the polynomial ring in $d$ variables, $d>1$.

\subsection{The normalized kernel} 
We begin by recalling some notions from complex geometry. Let $\LL$ be a holomorphic Hermitian line bundle over a complex manifold 
$\lO$. The Hermitian metric of $\LL$ is given by some smooth choice of an inner product $\|\.\|^2_w$ on the fibre $\LL_w$. There is a canonical (Chern) connection on $\LL$ which is compatible with both the Hermitian metric and the complex structure of $\LL$. The curvature $\lk$ of the line bundle $\LL$ on any fixed but arbitrary coordinate chart, with respect to the canonical connection, is given by the formula 
$$\lk(w):=-\dl\o\dl\log\|\lg(w)\|^2=-\S_{i,j}\dl_i\o\dl_j\log\|\lg(w)\|^2 dw_i\yi d\o w_j,$$
where $\lg$ is any non-vanishing holomorphic section of $\LL.$ Since any two such sections differ by multiplication by a non-vanishing holomorphic function, it is clear that the definition of the curvature is independent of the choice of the holomorphic section $\lg$. Indeed, it is well known that two such line bundles are locally equivalent if and only if their curvatures are equal. For holomorphic Hermitian vector bundles (rank $> 1$) the local equivalence involves not only the curvature but also its covariant derivatives, see \cite{CD}.  

In general, Lemma 2.3 of \cite{W} singles out a frame $\lg^{(0)}$ such that the metric has the form:  
$$\|\lg^{(0)}(w)\|^2=I+O(|w|^2)$$ 
and it follows that 
$$\lk(0)=\S_{i,j}\big(\dl_i\o\dl_j\|\lg^{(0)}(w)\|^2\big)_{|w=0}dw_i\yi d\o w_j.$$ 
In a slightly different language, a {\bf normalized kernel} $\KL^{(0)}$ at $w_0$ is defined in \cite[Remark 4.7(b)]{CS} by requiring that $\KL^{(0)}(z,w_0)\equiv I$. Setting $\lg^{(0)}(w)=\KL_w^{(0)}$, we see that the normalized kernel $\KL^{(0)}$ has no linear terms. 
Fix $w_0\in\lO$. There is a neighborhood, say $\lO_0$, of $w_0$ on which $\KL(z,w_0)$ doesn't vanish (for $n=1$) or is an invertible
$n\xx n$-matrix (for $n>1$). Set
$$\lF_\lG^{(0)}(z)=\KL(w_0,w_0)^{1/2}\ \KL(z,w_0)^{-1},\,\, z\in\lO_0.$$ 
Then 
$$\KL^{(0)}(z,w):=\lF_\lG^{(0)}(z)\ \KL(z,w)\ \lF_\lG^{(0)}(w)^*$$ 
is a normalized kernel on $\lO_0$. Thus starting with an analytic Hilbert module $\ML$ possessing a reproducing kernel $\KL$, there is a Hilbert module $\ML^{(0)}$ possessing a normalized reproducing kernel $\KL^{(0)}$, isomorphic to $\ML$.  Now, it is evident that two Hilbert modules are isomorphic if and only if there is a unitary $U$ such that 
$$\KL^{(0)}_1(z,w)=U\ \KL^{(0)}_2(z,w)\ U^*.$$
In other words, the normalized kernel is uniquely determined up to a fixed unitary. In particular, if $n=1$, then the two Hilbert modules are isomorphic if and only if the normalized kernels are equal. We gather all this information in the following proposition.
 
\begin{proposition}\label{a} The following conditions on any pair of (scalar) analytic Hilbert modules over the polynomial ring are equivalent.
\begin{enumerate}
\item Two analytic Hilbert modules $\ML_1$ and $\ML_2$ are isomorphic.
\item The holomorphic line bundles $\LL_1$ and $\LL_2$ determined by the eigenspaces of the analytic Hilbert modules $\ML_1$ and 
$\ML_2$, respectively, are locally equivalent as Hermitian holomorphic bundles. 
\item The curvature of the two line bundles $\LL_i$, $i=1,2$, are equal.
\item The normalized kernels $\KL_i^{(0)}$, $i=1,2,$ at any fixed but arbitrary point $w_0$ are equal.
\end{enumerate}
\end{proposition}

\section{Invariants for submodules}\label{b}
In the paper \cite{CD1}, Cowen and Douglas pointed out that all submodules of the Hardy module $H^2(\Dl)$ are isomorphic. They used this observation to give a new proof of Beurling's theorem describing all invariant subspaces of $H^2(\Dl)$. Although all submodules of the Hardy module $H^2(\Dl)$ are isomorphic, the quotient modules are not. Surprisingly enough, this phenomenon distinguishes the multi-variable situation from the one variable case. Consider for instance the submodule $H_{(0,0)}^2(\Dl^2)$ of all functions vanishing at 
$(0,0)$ in the Hardy space $H^2(\Dl^2)$ over the bidisk $\Dl^2$. Then the module tensor product of $H_{(0,0)}^2(\Dl^2)$ over the polynomial ring $\Cl[\zl]$ in two variables with the one dimensional module $\Cl_w$, $(p,w)\mapsto p(w)$, is easily seen to be 
\be{8}H_{(0,0)}^2(\Dl^2)\xt_{\Cl[\zl]}\Cl_w=\begin{cases}\Cl\op\Cl&\mbox{if~}w=(0,0)\\\Cl &\mbox{if~}w\ne(0,0)\end{cases}\ee
while $H^2(\Dl^2)\xt_{\Cl[\zl]}\Cl_w=\Cl$. It follows that the submodule $H_{(0,0)}^2(\Dl^2)$ is not isomorphic to the module 
$H^2(\Dl^2),$ in stark contrast to the case of one variable. 

The existence of non-isomorphic submodules of the Hardy module $H^2(\Dl^2)$ indicates that inner functions alone may not suffice to  characterize submodules in this case. It is therefore important to determine when two submodules of the Hardy module, and also more general analytic Hilbert modules, are isomorphic. This question was considered in \cite{BCL} for the closure of some ideals 
$\IL\subseteq\Cl[\zl]$ in the Hardy module $H^2(\Dl^2)$ with the common zero set $\{(0,0)\}$. It was extended to a much larger class of ideals in the paper \cite{ACD}. A systematic study in a general setting culminated in the paper \cite{DPSY} describing a 
{\bf rigidity phenomenon} for submodules of analytic Hilbert modules in more than one variable. A different proof of the Rigidity Theorem using the sheaf model was given in \cite{SBisM}. A slightly different approach to obtaining invariants by resolving the singularity at $(0,0)$ was initiated in \cite{DMV}, and considerably expanded in \cite{SBisM}. We describe this approach briefly. 

A systematic study of Hilbert submodules of analytic Hilbert modules was initiated in the papers \cite{BMP,SBisM}. If $\IL$ is an ideal in $\Cl[\zl]$, consider the submodule $\wt\ML=[\IL]$ in an analytic Hilbert module $\ML\subseteq\x{Hol}(\lO,\Cl)$ obtained by taking the closure of $\IL.$ Let
$$\lO_\IL:=\{z\in\lO:\ f(z)=0\ \forall\ f\in\IL\}$$
denote the algebraic subvariety of $\lO$ determined by $\IL.$ For the reproducing kernel $\KL(z,w)$ of $\ML,$ the vectors $\KL_w\in\ML$ will in general not belong to the submodule $\wt\ML.$ However, one has a {\bf truncated kernel} $\wt\KL(z,w)=\wt\KL_w(z)$ such that 
$\wt\KL_w\in\wt\ML$ for all $w\in\lO,$ which induces a holomorphic Hermitian line bundle $\t\LL$ defined on $\lO\setminus\lO_\IL,$ with fibre
$$\t\LL_w=\x{Ran}\ \wt\KL_w,\ w\in\lO\setminus\lO_\IL,$$
and positive definite metric $\wt\KL(w,w).$ This line bundle $\t\LL$ does not necessarily extend to all of $\lO.$ In fact, on the singular set $\lO_\IL$ the eigenspace of the submodule $\wt\ML$ will in general be higher dimensional. However, in the paper \cite{SBisM}, using the monoidal transform, a line bundle $\h\LL$ was constructed on a certain blow-up space $\h\lO,$ with a holomorphic map $\lp:\h\lO\to\lO.$ (Actually, this construction holds locally, near any given point $w_0\in\lO_\IL.$) The restriction of this line bundle to the exceptional set $\lp^{-1}(\lO_\IL)$ in the blow-up space was shown to be an invariant for the submodule $\wt\ML$. 

For the submodule $\wt\ML=H_{(0,0)}^2(\Dl^2)\subseteq H^2(\Dl^2)$ of the Hardy module, corresponding to the point singularity 
$(0,0)\in\lO:=\Dl^2,$ the above construction can be made very explicit: The eigenspace of $\wt\ML$ at $w:=(w_1,w_2)\ne(0,0)$ is the one dimensional space spanned by the truncated kernel vector 
\be{2}\wt\KL_w(z):=\f1{(1-\o w_1 z_1)(1-\o w_2 z_2)}-1=\f{\o w_1 z_1+\o w_2 z_2-\o w_1 z_1\o w_2 z_2}{(1-\o w_1 z_1)(1-\o w_2 z_2)}.\ee
At $(0,0),$ this vector is the zero vector while the eigenspace of $\wt\ML$ is two dimensional, spanned by the vectors $z_1$ and $z_2$. We observe, however, that for $j=1,2$ the limit $\f{\wt\KL_w(z)}{w_j}$, along lines through the origin as $w\to 0,$ exists and is non-zero. Parametrizing the lines through $(0,0)$ in $\Dl^2$ by $w_2=\lt_1w_1$ or $w_1=\lt_2 w_2$, we obtain the coordinate charts for the Projective space $\Pl^1(\Cl)$. On these, we have
$$\lim_{\o w_2=\lt_1\o w_1,\,\o w\to 0}\f{\wt\KL_w(z)}{\o w_1}=z_1+\lt_1 z_2.$$
Similarly, we have 
$$\lim_{\o w_1=\lt_2\o w_2,\,\o w\to 0}\f{\wt\KL_w(z)}{\o w_2}=z_2+\lt_2 z_1.$$
Setting $s(\lt_1):=z_1+\lt_1 z_2$ and $s(\lt_2)=z_2+\lt_2 z_1$ taking values in $H_{(0,0)}^2(\Dl^2)$, we obtain a holomorphic Hermitian line bundle $\h\LL$ over projective space $\Pl^1(\Cl)$. The metric of this line bundle is given by the formula 
$$\|s(\lt_j)\|^2_{\wt\ML}=1+|\lt_j|^2$$
for $j=1,2$. It is shown in \cite[Theorem 5.1]{DMV}, see also \cite[Theorem 3.4]{SBisM}, that for many submodules of analytic Hilbert modules, the class of this holomorphic Hermitian line bundle on the projective space is an invariant for the submodule. Since the curvature is a complete invariant, it follows that in our case the curvature  
$$\lk(\lt_j)=(1-|\lt_j|^2)^{-2} d\lt_j\yi d\o\lt_j$$
for the coordinate $\lt_j$ ($j=1,2$) is an invariant for the submodule $H_{(0,0)}^2(\Dl^2)$.

Often it is possible to determine when two submodules of an analytic Hilbert module are isomorphic without explicitly computing a set of invariants.  A particular case is the class of submodules in an analytic Hilbert modules which are obtained by taking the closure of an ideal in the polynomial ring. Here the surprising discovery is that many of these submodules are isomorphic if and only if the ideals are equal. Of course, one must impose some mild condition on the nature of the ideal. For instance, principal ideals have to be excluded. Several different hypotheses that make this "rigidity phenomenon" possible are discussed in Section 3 of \cite{DPSY}. One of these is the theorem of \cite[Theorem 3.6]{DPSY}. A slightly different formulation given below is Theorem 3.1 of \cite{SBisM}.  

Let $\lO\ic\Cl^d$ be a bounded domain. For $k=1,2,$ let $[\IL_k]$ be the closure in an analytic Hilbert module 
$\ML\subseteq\x{Hol}(\lO)$ of the ideal $\IL_k\subseteq\Cl[\zl]$.
  
\begin{theorem}[Theorem 3.1, \cite{SBisM}] \label{r} 
Assume that the dimension of $[\IL_k]/[\IL_k]_w$ is finite and that the dimension of the zero set of these modules is at most $d-2$. Also, assume that every algebraic component of $V(\IL_k)$ intersects $\lO$. Then $[\IL_1]$ and $[\IL_2]$ are isomorphic if and only if 
$\IL_1=\IL_2$.
\end{theorem}

In this paper we study submodules of (scalar valued) analytic Hilbert modules ($n=1$) which are related to higher-dimensional singularities. Starting with the weighted Bergman spaces defined on a bounded symmetric domain, the submodules are determined by a vanishing condition on the "Kepler variety". The new feature is that the singularity set is not a complete intersection (in the sense of algebraic geometry) which means that the usual projectivization involving monoidal transforms (blow-up process) is not sufficient for the resolution of singularities. We will replace it by a higher-rank blow-up process, having as exceptional fibres compact hermitian symmetric spaces of higher rank instead of projective spaces. The charts and analytic continuation we use are adapted to the geometry of the Kepler variety. The simplest case of rank 1 reduces to the usual blow-up process. 

In this setting we again obtain a rigidity theorem which is not a special case of Theorem \ref{r}, since we do not consider different ideals (i.e. different subvarieties) for the singular modules, but we consider a fixed subvariety and vary the underlying "big" Hilbert module, by choosing an arbitrary coefficient sequence or, as a special case, a $K$-invariant probability measure. This situation is most interesting in the symmetric case, where one has a full scale of different Hilbert modules like the weighted Bergman spaces. Then we show that the "truncated" kernel of the submodule can be recovered from the reduction to the blow-up space. This is a kind of rigidity in the parameter space instead of selecting different ideals. 

\section{Jordan-Kepler Varieties}
Hilbert modules and submodules defined by analytic varieties have been mostly studied for domains $\lO$ which are strongly pseudoconvex with smooth boundary, or a product of such domains. From an operator-theoretic point of view, this is natural since for strongly pseudoconvex (bounded) domains, Toeplitz operators with continuous symbols (in particular, with symbols given by the coordinate functions) are essentially normal, so that the Toeplitz $C^*$-algebra generated by such operators is essentially commutative and has a classical Fredholm and index theory. There are, however, interesting classes of bounded domains which are only weakly pseudoconvex (and are therefore domains of holomorphy, by the Cartan-Thullen theorem) with a non-smooth boundary. A prominent class of such domains are the {\bf bounded symmetric domains} of arbitrary rank $r,$ which generalize the (strongly pseudoconvex) unit ball, having rank $r=1.$ The Hardy space and the weighted Bergman spaces of holomorphic functions on bounded symmetric domains have been extensively studied from various points of view (see, e.g., \cite{BBCZ,FK1,U}. More recently, irreducible subvarieties of symmetric domains, given by certain determinant type equations, have been studied in \cite{EU} under the name of 'Jordan-Kepler varieties.' This terminology is used since the rank $r=2$ case corresponds to the classical Kepler variety in the cotangent bundle of spheres \cite{BEY}

In order to describe bounded symmetric domains and their determinantal subvarieties, we will use the {\bf Jordan theoretic} approach to bounded symmetric domains which is best suited for harmonic and holomorphic analysis on symmetric domains. For background and details concerning the Jordan theoretic approach, we refer to \cite{FK2,L,U}.    

Let $V$ be an irreducible hermitian Jordan triple of rank $r,$ with Jordan triple product denoted by $\{u;v;w\}.$ The so-called {\bf spectral unit ball} $\lO\ic V$ is a bounded symmetric domain. Conversely, every (irreducible) bounded symmetric domain can be realized in this way. An example is the matrix space $V=\Cl^{r\xx s}$ with triple product
$$\{u;v;w\}:=uv^*w+wv^*u,$$
giving rise to the matrix ball
$$\lO=\{z\in\Cl^{r\xx s}:\ I_r-zz^*>0\}.$$
In particular, for rank $r=1$ we obtain the triple product
$$\{u;v;w\}:=(u|v)w+(w|v)u$$
on $V=\Cl^d,$ with inner product $(u|v),$ giving rise to the unit ball
$$\lO=\{z\in\Cl^d:\ (z|z)<1\}.$$
Let $G$ denote the identity component of the full holomorphic automorphism group of $\lO$. Its maximal compact subgroup
$$K:=\{k\in G:\ k(0)=0\}$$
consists of linear transformations preserving the Jordan triple product. For $z,w\in V$ define the {\bf Bergman operator} $B_{z,w}$ acting on $V$ by
$$B_{z,w}v=v-\{z;w;v\}+\f14\{z\{w;v;w\}z\}.$$
We can also write 
\be{1}B_{z,w}=I-D(z,w)+Q_zQ_w,\ee
where
$$D(z,w)v=\{z;w;v\},$$
and 
$$Q_zw:=\{z;w;z\}$$
denotes the so-called quadratic representation (conjugate linear in $w$). For matrices, we have
$D(z,w)v=zw^*v+vw^*z,\ Q_zw=zw^*z$ and hence 
\be{13}B_{z,w}v=(1_r-zw^*)v(1_s-w^*z).\ee
An element $c\in V$ satisfying $c=Q_cc$ is called a {\bf tripotent}. For matrices these are the partial isometries. Any tripotent $c$ induces a {\bf Peirce decomposition} 
$$V=V_2^c\op V_1^c\op V_0^c.$$
We have 
$$d_\l:=\dim\r V_\l=d_2^c+d_1^c,$$
where
$$d_2^c=\dim V_2^c=\l(1+\f a2(\l-1)),$$
$$d_1^c=\dim V_1^c=\l(a(r-\l)+b).$$
Here $a,b$ are the so-called characteristic multiplicities defined in terms of a joint Peirce decomposition \cite{L}. Moreover,
$$\f{2d_2^c+d_1^c}{\l}=2(1+\f a2(\l-1))+a(r-\l)+b=2+a(r-1)+b=p$$
is the genus. As a fundamental property, there exists a {\bf Jordan triple determinant} 
\be{12}\lD:V\xx V\to\Cl,\ee 
which is a (non-homogeneous) sesqui-polynomial satisfying
$$\det B_{z,w}=\lD(z,w)^p.$$
For $(r\xx s)$-matrices, we have $p=r+s$ and 
$$\lD(z,w)=\det(1_r-zw^*)$$ 
as a consequence of \er{13}. In particular, $\lD(z,w)=1-(z|w)$ in the rank 1 case $V=\Cl^d.$ A hermitian Jordan triple $U$ is called 
{\bf unital} if it contains a (non-unique) tripotent $u$ such that $D(u,u)=2\.I.$ In this case $U$ becomes a Jordan *-algebra with unit element $u$ under the multiplication 
$$z\oc w:=\f12{z;u;w}$$ 
and involution 
$$z^*:=Q_uz=\f12\{u;z;u\}.$$ 
This Jordan algebra has a homogeneous {\bf determinant polynomial} $N:U\to\Cl$ defined in analogy to Cramer's rule for square matrices. Every Peirce 2-space $V_2^c$ is a unital Jordan triple with unit $c.$

Now we introduce certain $K$-invariant varieties. Every hermitian Jordan triple $V$ has a natural notion of {\bf rank} defined via spectral theory. For fixed $\l\le r$ let 
$$\r V_\l=\{z\in V:\ \x{rank}(z)=\l\}$$
denote the {\bf Jordan-Kepler manifold} studied in \cite{EU}. It is a $K^\Cl$-homogeneous manifold whose closure is the {\bf Jordan-Kepler variety}
$$V_\l=\{z\in V:\ \x{rank}(z)\le\l\}.$$
One can show that the smooth part of $V_\l$ (in the sense of algebraic geometry) is precisely given by $\r V_\l.$ Thus the 
{\bf singular points} of $V_\l$ form the closed subvariety $V_{\l-1},$ which has codimension $>1,$ unless we have the case 
$\l=r$ for tube domains ($b=0$). This case will be excluded in the sequel. The {\bf center} $S_\l\ic\r V_\l$ consists of all tripotents of rank $\l.$ 

\section{Hilbert modules on Kepler varieties}
Combining the Kepler variety and the spectral unit ball, we define the {\bf Kepler ball}
$$\lO_\l:=\lO\ui V_\l$$
for any $0\le\l\le r.$ The Kepler ball $\lO_\l$ has singularities exactly at $\lO_{\l-1},$ so that the smooth part of $\lO_\l$ is given by
$$\r\lO_\l:=\r V_\l\ui\lO_\l=\lO_\l\setminus\lO_{\l-1}.$$ 
Apart from the case $\l=r$ on tube type domains, which we exclude here, the singular set $\lO_{\l-1}\ic\lO_\l$ has codimension 
$>1.$ Combining this with the fact that $V_\l$ is a normal variety (so that the second Riemann extension theorem holds) it follows that every holomorphic function on $\r\lO_\l$ has a unique holomorphic extension to $\lO_\l.$ Henceforth we will identify holomorphic functions on $\r\lO_\l$ with their unique holomorphic extension to $\lO_\l.$ For any $K$-invariant measure 
$\lr$ on $\r V_\l$ we have a {\bf polar integration formula}
$$\I_{\r V_\l}d\lr(z)\ f(z)=\I_{\lL_2^c}d\lr^c(t)\I_K dk\ f(k\F t)$$
where $\lr^c$ is a measure on the symmetric cone $\lL_2^c$ of $V_2^c$ \cite{FK2} called the {\bf radial part} of $\lr.$ Here $\F t$ denotes the Jordan algebraic square root in $\lL_2^c.$ As a special case, consider the {\bf Riemann measure} $\ll_\l(dz)$ on $\r V_\l$ which is induced by the normalized inner product on $V.$ Denoting by $\lF_\l$ the Koecher-Gindikin Gamma function of $\lL_2^c$ 
\cite{FK2}, its polar decomposition is
\be{17}\I_{\r V_\l}\f{\ll_\l(dz)}{\lp^{d_\l}}\ f(z)=\f{\lG_\l(\f{a\l}2)}{\lG_\l(\f dr)\lG_\l(\f{ar}2)}\I_{\lL_2^c}dt\ N_c(t)^{d_1^c/\l}
\I_K dk\ f(k\F t).\ee
Here $N_c$ is the Jordan algebra determinant on $V_2^c$ normalized by $N_c(c)=1.$ For $\l=r$ the Riemann measure on the open dense subset $\r V_r=\r V\ic V$ agrees with the Lebesgue measure, and \er{17} gives the well-known formula
$$\I_V\f{dz}{\lp^d}\ f(z)=\f1{\lG(\f dr)}\I_{\lL_2^e}dt\ N_e(t)^b\I_K dk\ f(k\F t)$$
for any maximal tripotent $e\in S=S_r.$ As a consequence of \er{17} we have for the Kepler ball
\be{18}\I_{\r\lO_\l}\f{\ll_\l(dz)}{\lp^{d_\l}}\ \lD(z,z)^{\ln-p}\ f(z)
=\f{\lG_\l(\f{a\l}2)}{\lG_\l(\f dr)\lG_\l(\f{ar}2)}\I_{\lL_2^c\ui(c-\lL_2^c)}dt\ N_c(t)^{d_1^c/\l}\ N_c(c-t)^{\ln-p}\I_K dk\ f(k\F t)\ee
since $\lD(k\F t,k\F t)=\lD(\F t,\F t)=N_c(c-t)$ for all $t\in\lL_2^c\ui(c-\lL_2^c).$ 

As a fundamental fact \cite{FK2,U} of harmonic analysis on Jordan algebras and Jordan triples, the Fischer-Fock reproducing kernel 
$e^{(z|w)},$ for the normalized $K$-invariant inner product $(z|w)$ on $V,$ has a "Taylor expansion"
$$e^{(z|w)}=\S_\m E^\m(z,w)$$
over all integer partitions $\m=m_1\ge m_2\ge\;\ge m_r\ge 0,$ where $E^\m(z,w)=E_w(z)$ are sesqui-polynomials which are $K$-invariant such that the finite-dimensional vector space
$$\PL_\m(V)=\{E_w^\m:\ w\in V\}$$
is an irreducible $K$-module. These $K$-modules are pairwise inequivalent and span the polynomial algebra $\PL(V).$ Let
$$(\ln)_\m=\P_{j=1}^r(\ln-\f a2(j-1))_{m_j}$$
denote the multi-variable {\bf Pochhammer symbol}. Let $\Nl_+^r$ denote the set of all partitions of length $\le r.$ Restricted to the Kepler variety we only consider partitions in $\Nl_+^\l$ of length $\le\l,$ completed by zeroes at the end. 

\begin{lemma} For any partition $\m\in\Nl_+^\l$ of length $\le\l$ we have
\be{7}\I_{\lL_2^c\ui(c-\lL_2^c)}dt\ N_c(t)^{d_1^c/\l}\ N_c(c-t)^{\ln-p}\ N_\m(t)
=\f{\lG_\l(\f{d_\l}{\l})\ \lG_\l(\ln-\f{d_\l}{\l})}{\lG_\l(\ln)}\ \f{(d_\l/\l)_\m}{(\ln)_\m}.\ee
\end{lemma}
\begin{proof} Applying \cite[Theorem VII.1.7]{FK2} to $\lL_2^c$ yields
$$\I_{\lL_2^c\ui(c-\lL_2^c)}dt\ N_c(t)^{d_1^c/\l}\ N_c(c-t)^{\ln-p}\ N_\m(t)
=\f{\lG_\l(\m+\f{d_1^c}{\l}+\f{d_2^c}{\l})\ \lG_\l(\ln-p+\f{d_2^c}{\l})}{\lG_\l(\m+\ln-p+\f{d_1^c+2d_2^c}{\l})}$$
$$=\f{\lG_\l(\m+\f{d_\l}{\l})\ \lG_\l(\ln-\f{d_\l}{\l})}{\lG_\l(\m+\ln)}
=\f{\lG_\l(\f{d_\l}{\l})\ \lG_\l(\ln-\f{d_\l}{\l})}{\lG_\l(\ln)}\ \f{(d_\l/\l)_\m}{(\ln)_\m}.$$
\end{proof}

Let $du$ be the $K$-invariant probability measure on $S_\l$ and put
\be{5}(f|g)_{S_\l}=\I_{S_\l}du\ \o{f(u)}\ g(u)=\I_K dk\ \o{f(kc)}\ g(kc).\ee

\begin{definition} Consider a coefficient sequence $(\lr_\m)_{\m\in\Nl_+^\l}$ normalized by $\lr_0=1.$ Define a Hilbert space 
$\ML=\ML_\lr$ of holomorphic functions on $\lO_\l$ by imposing the $K$-invariant inner product
\be{4}(f|g)_\lr:=\S_{\m\in\Nl_+^\l}\lr_\m(f_\m|g_\m)_{S_\l}.\ee
where $f_\m\in\PL_\m(V)$ denotes the $\m$-th component of $f.$ 
\end{definition}

The {\bf subnormal case} arises when the inner product \er{4} has the form
$$(f|g)_\lr=\I d\lr(z)\ \o{f(z)} g(z),$$
where $\lr$ is a $K$-invariant probability measure on the closure of $\lO_\l$ or a suitable $K$-invariant subset which is a set of uniqueness for holomorphic functions. For the case $\l=r$, this was studied in detail for the tube type domains in \cite{BM} and completed for all bounded symmetric domains in \cite{AZ}. By \cite[Proposition 4.4]{EU} the Hilbert space 
$$\ML=\ML_\lr:=\{\lf\in L^2(d\lr):\ \lf\mbox{ holomorphic on }\lO_\l\}$$
has the coefficient sequence
$$\lr_\m=\I_{\lL_2^c}d\lr^c(t)\ N_\m(t)$$
given by the {\bf moments} of the radial part $\lr^c,$ which is a probability measure on $\lL_2^c$ (not necessarily of full support).  As a special case the Hardy type inner product \er{5}, corresponding to the $K$-invariant probability measure $du$ on $S_\l,$ has the point mass at $c$ as its radial part, showing that all radial moments $\lr_\m=1.$

It is clear that the Hilbert spaces $\ML_\lr$ defined by $K$-invariant measures are analytic Hilbert modules as defined above (however, consisting of holomorphic functions on a manifold $\r\lO_\l$ instead of a domain). For more general coefficient sequences $\lr_\m,$ one could in principle determine whether multiplication operators by polynomials are bounded (using certain growth conditions on the coefficient sequence), and whether the other requirements for analytic Hilbert modules hold. Important examples are listed below where the reproducing kernels are given by hypergeometric series. For the classical case $\l=r,$ the well-understood analytic continuation of the scalar holomorphic discrete series of weighted Bergman spaces on $\lO=\lO_r$ \cite{FK1} shows that the Hilbert module property extends beyond the subnormal case.  

\begin{proposition} For a given coefficient sequence $\lr_\m,$ $\ML$ has the {\bf reproducing kernel}
\be{6}\KL(z,w)=\S_{\m\in\Nl_+^\l}\f{(d/r)_\m}{\lr_\m}\ \f{(ra/2)_\m}{(\l a/2)_\m}\ E^\m(z,w).\ee
\end{proposition}
\begin{proof} This follows from \cite[Proposition 4.3]{EU} and the formula
$$\f{d_\m}{d_\m^c}=\f{(d/r)_\m}{(d_2^c/\l)_\m}\ \f{(ra/2)_\m}{(\l a/2)_\m}$$
obtained in \cite[equation (5.5) in the proof of Theorem 5.1]{EU}.
\end{proof}

We will now present some examples, where the reproducing kernel \er{6} can be expressed in closed form as a multivariate hypergeometric series defined in general by
$$\tensor[^{}_{p}]{\ba{\la_1,\;,\la_p}{\lb_1,\;,\lb_q}}{_q}(z,w)=\S_\m\f{(\la_1)_\m\:(\la_p)_\m}{(\lb_1)_\m\:(\lb_q)_\m}\ E^\m(z,w).$$
Applying \er{7} to $\m=0$ it follows that
$$\lr^\ln(dz)=\f{\lG_\l(\f dr)}{\lG_\l(\f{d_\l}{\l})}\ \f{\lG_\l(\f{ra}2)}{\lG_\l(\f{\l a}2)}\f{\lG_\l(\ln)}{\lG_\l(\ln-\f{d_\l}{\l})}
\ \f{\ll_\l(dz)}{\lp^{d_\l}}\ \lD(z,z)^{\ln-p}$$ 
is a probability measure on $\r\lO_\l.$ Moreover, applying \er{7} to any $\m\in\Nl_+^\l$ it follows that the measure $\lr^\ln$ has the coefficient sequence
$$\lr_\m^\ln=\ \f{(d_\l/\l)_\m}{(\ln)_\m}.$$
Thus the Hilbert space 
$$\ML_\ln:=\{\lf\in L^2(d\lr^\ln):\ \lf\mbox{ holomorphic on }\lO_\l\}$$
of holomorphic functions on $\lO_\l$ has the reproducing kernel
$$\KL(z,w)=\S_{\m\in\Nl_+^\l}\f{(d/r)_\m}{(d_\l/\l)_\m}\ \f{(ra/2)_\m}{(\l a/2)_\m}\ (\ln)_\m\ E^\m(z,w)
=\tensor[^{}_{3}]{\ba{\f dr,\ \f{ra}2,\ \ln}{\f{d_\l}{\l},\ \f{\l a}2}}{_2}(z,w).$$
In the classical case $\l=r$ we have the probability measure 
$$d\lr^\ln(z)=\f{\lG(\ln)}{\lG(\ln-\f dr)}\ \f{dz}{\lp^d}\ \lD(z,z)^{\ln-p}$$ 
on $\lO,$ whose reproducing kernel is given by
$$\KL(z,w)=\S_{\m\in\Nl_+^r}(\ln)_\m\ E^\m(z,w)=\tensor[^{}_{1}]{\ba{\ln}{}}{_0}(z,w)=\lD(z,w)^{-\ln}$$
according the Faraut-Kor\'anyi formula \cite{FK1}.

\section{The Singular Set and its Resolution}
The only strongly pseudoconvex symmetric domains are the unit balls of rank $r=1.$ Here the singularity $\lO_0$ consists of a single point $\{0\}.$ The classical procedure to resolve this singularity is the monoidal transformation (blow-up process) where a point is replaced by a projective space of appropriate dimension. As the main geometric result in this paper, we obtain a generalization of the blow-up process for higher dimensional Kepler varieties and domains of arbitrary rank. The Jordan theoretic approach leads to quite explicit formulas which generalize the equations of the classical blow-up process of a point.

The general procedure outlined in Section \ref{b} using monoidal transformations works in the case where the singularity is given by a regular sequence $g_1,\;,g_m$ of polynomials generating the vanishing ideal $\IL.$ In this case the variety is a smooth complete intersection. If $m=d$ equals the dimension, this variety reduces to a single point. The usual blow-up process around a point 
$0\in\Cl^d$ is the proper holomorphic map
$$\lp:\h\Cl^d\to\Cl^d$$
where
$$\h\Cl^d:=\{(w,U):\ w\in\Cl^d,\ U\in\Pl^{d-1},\ w\in U\}$$
is the tautological bundle over $\Pl^{d-1},$ with 'collapsing map' $\lp(w,U):=w.$ The map $\lp$ is biholomorphic outside the exceptional fibre $\lp^{-1}(0)=\Pl^{d-1}.$
For the Kepler varieties studied here the singular set $\lO_{\l-1}$ has higher dimension and is not a complete intersection (unless $\l=1$). Thus a regular generating sequence of polynomials does not exist. Instead, we use the harmonic analysis of polynomials provided by the Jordan theoretic approach to study the singular set. The main idea is to replace the projective space (a compact hermitian symmetric space of rank 1) by a compact hermitian symmetric space of higher rank, namely the {\bf Peirce manifold} 
$$M_\l=\{V_2^c:\ c\in S_\l\}$$ 
of all Peirce 2-spaces of rank $\l$ in $V.$ This can also be realized as the conformal compactification of the Peirce 1-space $V_c^1,$ for any rank $\l$ tripotent $c.$ For example, in the full matrix triple $V=\Cl^{r\xx s}$ the Peirce 1-space of $c=\bb{1_\l}000\in S_\l$ is given by
$$V_1^c=\bb0{\Cl^{\l\xx(s-\l)}}{\Cl^{(r-\l)\xx\l}}0.$$
Hence, in this case, the Peirce manifold $M_\l$ is the direct product of two Grassmann manifolds 
$$M_\l=\x{Grass}_\l(\Cl^r)\xx\x{Grass}_\l(\Cl^s).$$ 
In the simplest case $r=1$ we have $V=\Cl^d$ and for the tripotent $c=(1,0^{d-1})$ we have
$V_1^c=(0,\Cl^{d-1}).$ Its conformal compactification is $\h V_1^c=\Pl^{d-1},$ which is the exceptional fibre of the usual blow-up process for $0\in\Cl^d.$ More generally, for any non-zero tripotent $c$ we have $V_2^c=\Cl\.c$ and hence $V_1^c$ becomes the orhtogonal complement $c^\perp=\Cl^{d-1},$ with conformal compactification $\h V_1^c=\Pl^{d-1}.$ 

The standard charts of projective space $\Pl^{d-1}$ have the form 
$$\Lt_i:\Cl^{d-1}\to\Pl^{d-1},\quad\Lt_i(t_1,\;,\h t_i,\;,t_d):=[t_1:\;:1_i:\;:t_d]$$
using homogeneous coordinates on $\Pl^{d-1}.$ Note that for $1\le i\le d$, the rank 1 tripotent $c_i:=(0,\;,0,1,0,\;,0)\in\Cl^d$ has the Peirce 1-space 
$$V_1^{c_i}:=\{(t_1,\;,t_{i-1},0,t_{i+1},\;,t_d):\ (t_1,\;,\h t_i,\;,t_d)\in\Cl^{d-1}\}.$$
In the higher rank setting, the Bergman operators \er{1} serve to define canonical charts for the Peirce manifolds. For each tripotent 
$c\in S_\l$ and every $t\in V_1^c$ the transformation $B_{t,-c}\in K^\Cl$ preserves the rank. It follows that $B_{t,-c}c\in\r V_\l$ has a Peirce 2-space denoted by $[B_{t,-c}c].$ As shown in \cite{S} the map
\be{18}\Lt_c:V_1^c\to M,\ \Lt_c(t):=[B_{t,-c}c]\ee
is a holomorphic chart of $M.$ The range of the chart $\Lt_c$ is
$$M_c:=\{U\in M:\ N_U(c)\ne 0\}.$$
Here $N_U:U\to\Cl$ denotes a Jordan algebra determinant of the Jordan triple $U$ which, as a Peirce 2-space, is of tube type. The Jordan determinant is only defined after choosing a maximal tripotent in $U$ as a unit element, but any two such determinant functions differ by a non-zero multiple. It is shown in \cite{S} that the local charts $\Lt_c$ of $M_\l,$ for different tripotents 
$c,c'\in S_\l,$ are compatible and hence form a holomorphic atlas on $M_\l.$ 

One can make the passage $z\mapsto [z]$ to the Peirce 2-space more explicit by introducing the so-called (Moore-Penrose) pseudo-inverse. Every element $z\in\r V_\l$ has a {\bf pseudo-inverse} $\t z\in\r V_\l$ determined by the properties
$$Q_z\t z=z,\ Q_{\t z}z=\t z,\ Q_z\ Q_{\t z}=Q_{\t z}\ Q_t.$$
Using the pseudo-inverse, the orthogonal projection onto the Peirce 2-space of $V_2^z$ can be explicitly written down. 

\begin{lemma}\label{f} The pseudo-inverse of $B_{t,-c}c$ is given by
$$\t\Lt_c(t)=B_{t,-c}B_{t,-t}^{-1}c.$$
Thus the associated Peirce 2-space is
$$[\Lt_c(t):\t\Lt_c(t)]\in\h V_1^c\ic\h V.$$
\end{lemma}

Combining these remarks, the chart \er{18} can be written down explicitly. It is also instructive to embed $M_\l$ into the conformal compactification $\h V$ of the underlying Jordan triple $V$ (the compact hermitian symmetric space that is dual to the spectral unit ball $\lO$). According to \cite{L} $\h V$ can elegantly be described using a certain equivalence relation $[z:w]$ for pairs $z,w\in Z.$As shown in \cite{S}, one may identify the Peirce 2-space $V_2^z$ with the equivalence class $[z;\t z]\in\h V.$ Thus the local chart 
\er{18} associated to a tripotent $c\in S_\l$ can also be expressed via the embedding
$$\Lt_c:V_1^c\to M_\l\ic\h V$$
given by
$$\Lt_c(t)=[z;\t z],$$
where $z:=B_{t,-c}c\in\r V_\l$ and $\t z$ is computed via Lemma \ref{f}. In the sequel these more refined descriptions of the local charts will not be needed.

Having found the exceptional fibre $M_\l$ for the higher-rank blow-up process, we now consider the {\bf tautological bundle}
$$\h V_\l=\{(w,U)\in V\xx M_\l:\ w\in U\}\ic V_\l\xx M_\l$$
over $M_\l,$ together with the {\bf collapsing map}
$$\lp:\h V_\l\to V_\l,\quad\lp(w,U):=w$$
whose range is $V_\l.$ In \cite{EU} this map is used to show that $V_\l$ is a {\bf normal variety}. This property implies the so-called second Riemann extension theorem for holomorphic functions, of crucial importance in the following.
For each $s\in V_2^c$ the rank $\l$ element
\be{19}\ls_c(s,t):=B_{t,-c}s\ee
has the same Peirce 2-space $\Lt_c(t)$ as $B_{t,-c}c.$ We define a local chart
$$\lr_c:V_2^c\xx V_1^c\to\h V_\l$$ 
by
\be{10}\lr_c(s,t):=(\ls_c(s,t),\Lt_c(t))\ee
By \er{19} the range of the chart $\lr_c$ is
$$\h V_\l^c:=\{(w,U)\in\h V_\l:\ U\in\x{Ran}\ \Lt_c\}=\{(w,U)\in\h V_\l:\ N_U(c)\ne 0\}.$$
One shows that the charts $\lr_c,$ for $c\in S_\l,$ define a holomorphic atlas on $\h V_\l,$ such that the collapsing map 
$\lp:\h V_\l\to V_\l$ is holomorphic and is biholomorphic outside the singular set. We call $\h V_\l,$ together with the collapsing map the {\bf (higher rank) blow-up} of $V_\l.$ 

\begin{proposition} For rank 1, let $c:=(1,0).$ Then
$$\lr_c(s,t):=((s,st),[1:t])=((s,st),\Cl(1,t)),$$
where $s\in\Cl$ and $t\in\Cl^{d-1}.$ Here $[s:t]=[s:t_1:\;:t_{d-1}]$ denotes the homogeneous coordinates in $\Pl^{d-1}.$
\end{proposition}
\begin{proof} Clearly, $V_2^c=\Cl\.c=(\Cl,0)=[1:,0]$ and $V_1^c=(0,\Cl^{d-1}).$ Then
$$\ls_c(s,t)=B_{t,-c}s=\(1+(0,t)\ba10\)(s,0)\(\bb1001+\ba10\ab0 t\)=(s,0)\bb1 t 01=(s,st).$$
In particular, $\ls_c(1,t)=(1,t)$ has the Peirce 2-space $\Lt_c(t)=\Cl\.(1,t)=[1:t].$ It follows that
$$\lr_c(s,t)=(\ls_c(s,t),\Lt_c(t))=((s,st),\Cl\.(1,t))=((s,st),[1:t]).$$
\end{proof}

More generally, taking for $c=e_i$ the $i$-th basis unit vector ($1\le i\le d$) we obtain local charts
$$\lr_i(\lz^i,\lz')=((\lz^i,\lz^i\lz'),\Cl(1^i,\lz'))=((\lz^i,\lz^i\lz'),[1^i:\lz'])$$
where $\lz'=(\lz^j)_{j\ne i}.$ The finitely many charts $\lr_i$ ($1\le i\le d$) form already a covering of $Q.$ Using the grid approach to Jordan triples one can similarly choose finitely many charts in the general case. However, for many arguments using $K$-invariance it is more convenient to take the continuous family of charts $(\ls_c)_{c\in S_\l}.$

Since the analytic Hilbert modules considered here are supported on the Kepler ball $\lO_\l=\lO\ui V_\l$ we restrict the tautological bundle to the open subset
$$\h\lO_\l:=\{(w,U)\in\h V_\l:\ w\in\lO\}$$
and obtain a collapsing map $\lp:\h\lO_\l\to\h\lO_\l$ by restriction. The main idea to study singular submodules $\wt\ML$ is now to construct a hermitian holomorphic line bundle $\h\LL$ over $\h\lO_\l,$ whose curvature will be the crucial invariant of $\wt\ML.$ 

\begin{proposition}\label{c} There exists a holomorphic line bundle $\h\LL$ on $\h\lO_\l$ consisting of all equivalence classes
\be{11}[s,t,\ll\ \o{N_c(s)}]_c=\[s',t',\ll\ \o{N_{c'}(s')}\]_{c'}\ee
with $\ll\in\Cl.$ Here $c,c'\in S_\l$ are tripotents such that 
\be{3}\lr_c(s,t)=\lr_{c'}(s',t')\ee
for $(s,t)\in V_2^c\xx V_1^c$ and $(s',t')\in V_2^{c'}\xx V_1^{c'}.$ 
\end{proposition}
\begin{proof} The condition \er{3} implies $\ls_c(s,t)=\ls_{c'}(s',t')$ and 
$[\ls_c(1,t)]=\Lt_c(t)=\Lt_{c'}(t')=[\ls_{c'}(1,t')].$ This implies that $N_c(s)$ and $N_{c'}(s')$ do not vanish. Since the quotient maps $\f{\o{N_{c'}(s')}}{\o{N_c(s)}}$ satisfy a cocycle property, it follows that  
$$[s,t,\ll]_c=\[s',t',\ll\f{\o{N_{c'}(s')}}{\o{N_c(s)}}\]_{c'}$$
defines an equivalence relation yielding a holomorphic line bundle.
\end{proof}

At this point we do not fix a hermitian metric the line bundle $\LL$ over $\h D_\l.$ The metric depends on the choice of singular submodules $\wt\ML$ which will be defined below.

\section{Singular Hilbert Submodules}
Consider the partition
$$\je:=(1,\;,1,0,\;,0)$$
of length $\l,$ with 1 repeated $\l$ times. Given the Hilbert module $\ML=\ML_\lr$ as above, consider the $K$-invariant Hilbert submodule 
$$\wt\ML=\{\lq\in\ML:\ \lq|_{V_{\l-1}}=0\}.$$ 
The formula \er{6} yields the {\bf truncated kernel} in the form
\be{16}\wt\KL(z,w)=\S_{\m\in\Nl_+^\l}\f{(d/r)_{\m+\je}}{\lr_{\m+\je}}\ \f{(ra/2)_{\m+1}}{(\l a/2)_{\m+\je}}\ E^{\m+\je}(z,w),\ee
corresponding to vanishing of order $\ge 1$ on $V_{\l-1}.$ Using the identity
$$(\ln)_{\m+\je}=(\ln+1)_\m\ (\ln)_\je$$
one can also express this using Pochhammer symbols for $\m$ instead of $\m+\je.$
 
\begin{lemma}\label{g} Let $V$ be a unital Jordan triple, with Jordan algebra determinant $N.$ Then we have
$$E^{\m+\je}(z,w)=\f{(d/r)_\m}{(d/r)_{\m+\je}}\ N(z)\o{N(w)}\ E^\m(z,w).$$
\end{lemma}
\begin{proof} For tube type we have
$$E^\m(e,e)=\f{d_\m}{(d/r)_\m}.$$
Writing
$$E^{\m+\je}(z,w)=c_\m\ N(z)\o{N(w)}\ E^\m(z,w)$$
it follows that
$$\f{d_{\m+\je}}{(d/r)_{\m+\je}}=E^{\m+\je}(e,e)=c_\m\ E^\m(e,e)=c_\m\ \f{d_\m}{(d/r)_\m}.$$
Since $d_{\m+\je}=d_\m$ in the unital case, it follows that
$$c_\m=\f{(d/r)_\m}{(d/r)_{\m+\je}}.$$
\end{proof}

\begin{lemma} For $\m\in\Nl_+^\l$ we have for $s\in V_2^c$ and $t\in V_1^c$
$$E^{\m+\je}(z,B_{t,-c}s)=\f{(d_2^c/\l)_\m}{(d_2^c/\l)_{\m+\je}}\ N_c(P_cB_{t,-c}^*z)\ \o{N_c(s)}\ E^\m(z,B_{t,-c}s).$$
\end{lemma}
\begin{proof} Applying Lemma \ref{g} to the tube type Peirce 2-space $V_2^c$ of rank $\l$ implies
$$E^{\m+\je}(z,B_{t,-c}s)=E^{\m+\je}(B_{t,-c}^*z,s)=E_c^{\m+\je}(P_cB_{t,-c}^*z,s)$$
$$=\f{(d_2^c/\l)_\m}{(d_2^c/\l)_{\m+\je}}\ N_c(P_cB_{t,-c}^*z)\ \o{N_c(s)}\ E_c^\m(P_cB_{t,-c}^*z,s).$$
Since $E_c^\m(P_cB_{t,-c}^*z,s)=E^\m(B_{t,-c}^*z,s)=E^\m(z,B_{t,-c}s),$ the assertion follows.
\end{proof}

Since the truncated kernel $\wt\KL$ of $\wt\ML$ vanishes on the singular set $V_{\l-1}$ it cannot be used directly to define a hermitian line bundle over $V_{\l-1}.$ Instead, we first consider the module tensor product of $H_0^2(\lO_\l)$ over the polynomial ring $\PL(V)$  with the one dimensional module $\Cl_w$, $(p,w)\mapsto p(w).$ Similar as in \er{8} we have, as a consequence of \er{16} 
$$H_0^2(\lO_\l)\xt_{\PL(V)}\Cl_w=\begin{cases}\Cl &\mbox{if~}w\in\r\lO_\l\\\PL_\je(V)&\mbox{if~}w\in\lO_{\l-1}\end{cases}.$$
Here $\PL_\je(V)$ is the finite-dimensional $K$-module belonging to the partition $\je.$ The $K$-module $\PL_\je(V)$ has dimension $>1$ (since we exclude the case $\l=r$ for tube type, where $\PL_\je(V)$ is spanned by the Jordan algebra polynomial $N$). The ideal $\IL$ associated to the variety $V_{\l-1}$ is generated by $\PL_\je(V).$ For each $w\in\lO_\l$ there is a 'cross-section' $\PL_\je(V)\to H_0^2(\lO_\l)$ given by
$$p(z)\mapsto p(z)\.\lQ_w(z)$$ 
where
\be{15}\lQ(z,w)=\h\KL_w(z)=\S_{\m\in\Nl_+^\l}\f{(d/r)_{\m+\je}}{\lr_{\m+\je}}\ \f{(ra/2)_{\m+1}}{(\l a/2)_{\m+\je}}
\ \f{(d_2^c/\l)_\m}{(d_2^c/\l)_{\m+\je}}\ E^\m(z,w).\ee
Then $\lQ_w(z)\in\ML$ for each $w\in\lO_\l.$ Let $N_i,\ i\in I$ be an orthonormal basis of $\PL_\je(V).$ Then there is a holomorphic vector subbundle $\EL\ic\lO_\l\xx\ML$ over the Kepler ball $\lO_\l,$ whose fibre at $w\in V_\l$ is the span
$$\EL_w:=\<N_i(z)\ \lQ_w(z):\ i\in I\>=\PL_\je(V)\.\lQ_w\ic\ML.$$
The vector bundle $\EL$ is independent of the choice of orthonormal basis $N_i.$ Consider the pull-back vector bundle 
$$\xymatrix{\lp^*\EL\ar[d]&\EL\ar[d]\\\h\lO_\l\ar[r]_{\lp}&\lO_\l}$$ 
over $\h\lO_\l,$ under the collapsing map $\lp.$
We note that the 'canonical' choice of higher rank vector bundle $\EL$ over $\lO_\l,$ with typical fibre $\PL_\je(V)$ associated with the quotient module, is only possible for irreducible domains. In the reducible case \er{2} of the bidisk there is no natural choice of a rank 2 vector bundle having the fibre $<z_1,z_2>$ at the origin.

\begin{proposition}\label{d} For all $(s,t)\in V_2^c\op V_1^c$ we have
$$\wt\KL(z,B_{t,-c}s)=N_c(P_cB_{t,-c}^*z)\ \o{N_c(s)}\ \lQ(z,B_{t,-c}s).$$
\end{proposition}
\begin{proof} This follows from the computation
$$\wt\KL(z,B_{t,-c}s)=\S_{\m\in\Nl_+^\l}\f{(d/r)_{\m+\je}}{\lr_{\m+\je}}\ \f{(ra/2)_{\m+1}}{(\l a/2)_{\m+\je}}\ E^{\m+\je}(z,B_{t,-c}s)$$
$$=\S_{\m\in\Nl_+^\l}\f{(d/r)_{\m+\je}}{\lr_{\m+\je}}\ \f{(ra/2)_{\m+1}}{(\l a/2)_{\m+\je}}
\ \f{(d_2^c/\l)_\m}{(d_2^c/\l)_{\m+\je}}\ N_c(P_cB_{t,-c}^*z)\ \o{N_c(s)}\ E^\m(z,B_{t,-c}s)$$
$$=N_c(P_cB_{t,-c}^*z)\ \o{N_c(s)}\S_{\m\in\Nl_+^\l}\f{(d/r)_{\m+\je}}{\lr_{\m+\je}}\ \f{(ra/2)_{\m+1}}{(\l a/2)_{\m+\je}}
\ \f{(d_2^c/\l)_\m}{(d_2^c/\l)_{\m+\je}}\ E^\m(z,B_{t,-c}s).$$
\end{proof}

Now consider the holomorphic line bundle $\h\LL$ over the blow-up space $\h\lO_\l$ defined in Proposition \ref{c}.

\begin{theorem} There exists an anti-holomorphic embedding $\h\LL\ic\lp^*\EL,$ defined on each fibre 
$\h\LL_{w,U}\ic(\lp^*\EL)_{w,U}=\EL_w$ by
\be{9}[s,t,1]_c\mapsto N_c(B_{t,-c}^*z)\ \lQ_{B_{t,-c}s}(z).\ee
In short,
$$[s,t,1]_c\mapsto N_c\oc B_{t,-c}^*\ \lQ_{B_{t,-c}s}.$$
\end{theorem}
\begin{proof} First we show that the map \er{9} is well-defined via the local charts \er{10}. Suppose that $c,c'\in S_\l$ satisfy
$$\lr_c(s,t)=\lr_{c'}(s',t'),$$
where $(s,t)\in V_2^c\xx V_1^c$ and $(s',t')\in V_2^{c'}\xx V_1^{c'}.$ Then we have
$$B_{t,-c}s=\ls_c(s,t)=\ls_{c'}(s',t')=B_{t',-c'}s'.$$
It follows that $\wt\KL_{B_{t,-c}s}=\wt\KL_{B_{t',-c'}s'}$ and Proposition \ref{d} implies
$$\o{N_c(s)}\ [s,t,1]_c=\wt\KL_{B_{t,-c}s}=\wt\KL_{B_{t',-c'}s'}=\o{N_{c'}(s')}\ [s',t',1]_{c'}.$$
Since $N_c(s)$ and $N_{c'}(s')$ don't vanish on the overlap $V_c\ui V_{c'}$ it follows that
$$[s,t,1]_c=\f{\o{N_{c'}(s')}}{\o{N_c(s)}}\ [s',t',1]_{c'}=\[s',t',\f{\o{N_{c'}(s')}}{\o{N_c(s)}}\]_{c'}.$$
Thus the map \er{9} respects the equivalence relation \er{11}. Moreover, the map \er{9} is anti-holomorphic in $(s,t),$ with values in 
$\ML.$ In order to see that the range belongs to the span of $N_i(z)\ \lQ_w(z),$ where $w=B_{t,-c}s,$ choose holomorphic functions 
$c_i(t)$ such that
$$N_c(B_{t,-c}^*z)=\S_{i\in I}\o{c_i(t)}\ N_i(z)$$
for all $t\in V_1^c.$ It follows that
$$N_c(B_{t,-c}^*z)\ \lQ_{B_{t,-c}s}(z)=\S_i N_i(z)\ \o{c_i(t)}\ \lQ(z,B_{t,-c}s)\in\EL_{B_{t,-c}s}.$$
\end{proof}

We are now able to define a {\bf hermitian metric} on the line bundle $\h\LL$ over $\h\lO_\l.$ A Jordan theoretic argument yields

\begin{lemma}\label{e} For $t\in V_1^c$ we have
$$P_cB_{t,-c}^*B_{t,-c}c=P_cB_{t,-t}c$$
and hence
$$N_c(B_{t,-c}^*B_{t,-c}c)=\lD(t,t).$$
Here $\lD$ denotes the Jordan triple determinant \er{12}.
\end{lemma}

\begin{proposition}\label{h} For all $(s,t)\in V_2^c\op V_1^c$ we have
$$\wt\KL(B_{t,-c}s,B_{t,-c}s)=\lD(t,t)\ |N_c(s)|^2\ \lQ(B_{t,-c}s,B_{t,-c}s).$$
\end{proposition}
\begin{proof} Since $P_cB_{t,-c}B_t^*P_c$ belongs to the structure group of $V_c^2$ it follows from Lemma \ref{e} that
$$=N_c(B_{t,-c}^*B_{t,-c}s)=N_c(B_{t,-c}^*B_{t,-c}c)\ N_c(s)=\lD(t,t)\ N_c(s).$$
Now apply Proposition \ref{d}.
\end{proof}

\begin{proposition} For each submodule $\wt\ML\ic\ML,$ with truncated kernel \er{16}, there exists a hermitian metric on the line bundle 
$\h\LL$ over $\h\lO_\l,$ given by the local representatives
$$([s,t,1]_c|[s,t,1]_c):=\lD(t,t)\ \lQ(B_{t,-c}s,B_{t,-c}s).$$
For this metric, the embedding \er{9} is isometric.
\end{proposition}
\begin{proof} Since Proposition \ref{h} implies
$$\|N_c(B_{t,-c}^*z)\ \lQ_{B_{t,-c}s}(z)\|^2=\|\f{\wt\KL_{B_{t,-c}s}}{\o{N_c(s)}}\|^2=\f1{|N_c(s)|^2}\ \wt\KL(B_{t,-c}s,B_{t,-c}s)
=\lD(t,t)\ \lQ(B_{t,-c}s,B_{t,-c}s)$$
it follows that the embedding \er{9} is isometric.
\end{proof}

\begin{definition} The Hilbert module over $\h\lO_\l$ associated with the hermitian holomorphic line bundle $\h\LL$ will be called the 
{\bf reduction} of $\wt\ML,$ and denoted by $\h\ML.$ Note that this is different from the pull-back $\lp^*\EL$ which is a vector bundle containing $\h\LL$ as a subbundle.
\end{definition}

The following {\em rigidity theorem for singular submodules on Kepler varieties} is our main analytic result.

\begin{theorem} Consider two $K$-invariant Hilbert modules $\wt\ML_\lr$ and $\wt\ML_{\lr'}$ on $\lO_\l,$ for given coefficient sequences 
$\lr_\m$ and $\lr_\m',$ respectively. Suppose that the reduced Hilbert modules $\h\ML_\lr$ and $\h\ML_{\lr'}$ on the blow-up space 
$\h\lO_\l$ are equivalent. Then we have equality $\wt\ML_\lr=\wt\ML_{\lr'}.$
\end{theorem}
\begin{proof} The proof is an application of the 'normalized kernel argument' summarized in Proposition \ref{a}. Consider the reproducing kernels $\h\KL^\lr$ and $\h\KL^{\lr'}$ of the reduced Hilbert modules. It suffices to consider a local chart 
$V_2^c\xx V_1^c$ of $\h\lO_\l$ for a given tripotent $c\in S_\l$ defined in \er{10}. As a consequence of module equivalence for line bundles, there exists a non-vanishing holomorphic function $\lf$ on the local chart $V_2^c\xx V_1^c$ of $\h\lO_\l$ such that
\be{14}\h\KL^{\lr'}(x,y)=\lf(x)\ \h\KL^\lr(x,y)\ \o{\lf(y)}.\ee
Putting $y=0$ we obtain
$$1=\h\KL^{\lr'}(x,0)=\lf(x)\ \h\KL^\lr(x,0)\ \o{\lf(0)}=\lf(x)\ \o{\lf(0)}.$$
Therefore $\lf$ is constant. After normalization, we may assume $\lf=1.$ Then \er{14} implies
$$\h\KL^{\lr'}(x,y)=\h\KL^\lr(x,y)$$
for all $x,y.$ In view of \er{15}, this implies $\lr_{\m+\je}=\lr_{\m+\je}'$ for all $\m\in\Nl_+^\l.$ By \er{16}, the singular submodules $\wt\ML$ and $\wt\ML'$ have the same truncated kernel $\wt\KL(z,w)=\wt\KL'(z,w).$
\end{proof}

\section{Outlook and Concluding Remarks}
For the Hardy module $H^2(\Dl^d)$ it is evident that not all submodules are of the form $[\IL],$ for some ideal $\IL$ of the polynomial ring. (Here $[\IL]$ is the closure of $\IL$ in $H^2(\Dl^d)$). Ahern and Clark \cite{AC} show that all submodules (of the Hardy module) of finite codimension are of this form. In general, if a submodule $\wt\ML\subseteq\ML$ is not of the form $[\IL]$, then it is not covered by the known Rigidity theorems with only one exception, namely \cite[Theorem, pp. 70]{DPY}. 
However, the geometric invariants constructed in \cite{SBisM} and in the current paper, it is hoped, might be useful in studying a much larger class of submodules. Recall that a submodule of an analytic Hilbert module $\ML$ based on the domain $\lO$ defines a coherent analytic sheaf \cite{BMP,SBisM}. It possesses a Hermitian structure away from the zero variety and on this smaller open set, we have a holomorphic Hermitian vector bundle, which determines the class of the submodule. What we have shown here is that it has an analytic Hermitian continuation to the blow-up space. This interesting phenomenon naturally leads to the notion of, what one may call a Hermitian sheaf and eventually determine the equivalence class of these in terms of the geometric data already implicit in the definition, as in the examples we have discussed here.

We conclude this paper with several remarks concerning interesting directions for future research

\begin{remark} In \cite{MU2} we consider more general Hilbert modules related to Kepler varieties, where the integration does not take place on the Kepler ball $\lO_\l$ but on certain boundary strata, including the Hardy type inner product \er{5}. These Hilbert modules, and their submodules defined by a vanishing condition on $\lO_{\l-1}$ provide a wider class of natural examples to which the above treatment is applicable.   
\end{remark}

\begin{remark} It is easy to generalize the singular Hilbert modules treated in this paper, defined by a vanishing condition of order 1 on the singular set, to vanishing conditions of higher order. In this case the truncated kernel, generalizing \er{16}, has the form
$$\wt\KL(z,w)=\S_{\m\in\Nl_+^\l}\f{(d/r)_{\m+\k}}{\lr_{\m+\k}}\ \f{(ra/2)_{\m+\k}}{(\l a/2)_{\m+\k}}\ E^{\m+\k}(z,w),$$
corresponding to vanishing of order $\ge k$ on $V_{\l-1}.$ Here $\k=(k,\;,k,0,\;,0)$ with $k$ repeated $\l$ times. In principle, one could also start with an arbitrary partition $\lm>0$ of length $\l$ and consider truncations such as
$$\wt\KL(z,w)=\S_{\n\in\Nl_+^\l,\ \n\ge\lm}\f{(d/r)_\n}{\lr_\n}\ \f{(ra/2)_\n}{(\l a/2)_\n}\ E^\n(z,w).$$
In this case one expects to have the finite-dimensional $K$-module $\PL_\lm(V)$ occurring as a quotient module. On the other hand, treating singularities where the rank decreases by more than 1, for example $V_{\l-2}\ic V_\l,$ or the origin $V_0=\{0\}$ as a singularity in $\lO=\lO_r,$ seems to be more difficult.
\end{remark}

\begin{remark} In the maximal rank case $\l=r$ the ball $\lO_r=\lO$ is invariant under the full non-linear group $G.$ For tube type domains, the singular set $\lO_{r-1}$ has codimension 1, defined by vanishing of the Jordan algebra determinant. This case formally resembles the one-dimensional situation and is not covered by our approach (it was excluded to begin with). On the other hand, let $V$ be a hermitian Jordan triple not of tube type. There are three cases
\bi
\i{} The rectangular matrices $V=\Cl^{r\xx s}$ with $s>r.$
\i{} The skew-symmetric matrices $V=\Cl_{asym}^{N\xx N}$ of odd order $N=2r+1$
\i{} The exceptional Jordan triple $V=\Ol_\Cl^{1\xx 2}$ of rank $r=2$ and dimension $16.$ 
\ei
For these cases the singular set 
$$V_{r-1}=\{z\in V:\ \x{rank}(z)<r\}$$
has codimension $>1.$ The intersection
$$\lO_{r-1}:=V_{r-1}\ui\lO$$
with the unit ball $\lO\ic V$ is an analytic subvariety of $\lO.$ For any automorphism $g\in G=\x{Aut}(\lO)$ we obtain another subvariety $g(\lO_{r-1})\ic\lO.$ Since $G$ acts on the weighted Bergman spaces $\ML_\ln=H_\ln^2(\lO)$ one can consider submodules of 
$\ML_\ln$ defined by vanishing on $\lO_{r-1}$ and $g(\lO_{r-1}),$ respectively, where $g\in G$ does not belong to $K.$ 

A similar situation arises for the so-called {\bf Mok embeddings}
$$\li_c:B\to\lO$$
of the unit ball $B=\Bl_n$ into a symmetric domain $\lO$ of higher rank, constructed in \cite{UWZ}. Here $c\in S_1$ is any rank 1 tripotent. These embeddings have the property that the respective Bergman kernels satisfy 
$$\KL_B(x,y)=\KL_\lO(\li_c(x),\li_c(y))$$
for all $x,y\in B.$ Let $B_c:=\li_c(B)\ic\lO$ be the image variety (whose defining equations are explicitly known \cite{UWZ}) and consider, for $g\in G,$ the subvariety $g(B_c)$ with associated Hilbert submodule $\wt\ML_\ln\subseteq\ML_\ln$ defined by a vanishing condition on $g(B_c).$

It would be of interest to study the reduced modules and rigidity problems for singular submodules in such a $G$-equivariant setting. 
\end{remark}

\begin{remark} Beyond the scalar case treated in this paper, analytic Hilbert modules for higher rank vector bundles ($n>1$) have recently attracted much attention \cite{KM1,KM2,MU1} and should give rise to interesting singular submodules as well.
\end{remark}

\end{document}